\documentclass[10pt]{amsart} 

\usepackage[utf8]{inputenc} 

\usepackage{lipsum}

\usepackage{geometry} 
\usepackage{graphicx} 

\usepackage[english]{babel}

\usepackage{booktabs} 
\usepackage{array} 
\usepackage{paralist} 
\usepackage{verbatim} 
\usepackage{subfig} 

\usepackage{caption}

\usepackage{amsfonts,amsmath,amssymb,bbm,amsthm,amsbsy}

\usepackage{mathtools}

\makeatletter							
\newtheorem*{rep@theorem}{\rep@title}
\newcommand{\newreptheorem}[2]{%
\newenvironment{rep#1}[1]{%
 \def\rep@title{#2 \ref{##1}}%
 \begin{rep@theorem}}%
 {\end{rep@theorem}}}
\makeatother

\theoremstyle{theorem}
\newtheorem{theorem}{Theorem}

\newtheorem{proposition}[theorem]{Proposition}
\newtheorem{corollary}[theorem]{Corollary}

\newtheorem{thm}{Theorem}[section]
\newtheorem{lemma}[thm]{Lemma}
\newtheorem{prop}[thm]{Proposition}

\newtheorem{claim}[thm]{Claim}

\newtheorem*{thm*}{Theorem}
\newtheorem*{lemma*}{Lemma}
\newtheorem*{prop*}{Proposition}
\newtheorem*{corr*}{Corrolary}
\newtheorem*{claim*}{Claim}

\theoremstyle{remark}

\newtheorem{quest}[thm]{Question}

\newtheorem*{rmk*}{Remark}
\newtheorem*{conj*}{Conjecture}
\newtheorem*{quest*}{Question}

\theoremstyle{definition}
\newtheorem{defn}[thm]{Definition}
\newtheorem{exmp}[thm]{Example}

\newtheorem*{defn*}{Definition}
\newtheorem*{exmp*}{Example}

\newreptheorem{theorem}{Theorem}
\newreptheorem{corollary}{Corollary}
\newreptheorem{proposition}{Proposition}

\usepackage{fancyhdr} 
\usepackage{tikz}
\usepackage{tikz-cd}

\usepackage{IEEEtrantools}

\newenvironment{equ*}[1]{\begin{IEEEeqnarray*}{#1}}{\end{IEEEeqnarray*}}

\newcommand{\R}{\mathbb{R}}

\newcommand{\Q}{\mathbb{Q}}
\newcommand{\Z}{\mathbb{Z}}
\newcommand{\N}{\mathbb{N}}
\newcommand{\Scl}{\mathcal{S}}

\newcommand{\att}{\mathtt{a}}
\newcommand{\btt}{\mathtt{b}}
\newcommand{\ctt}{\mathtt{c}}
\newcommand{\ttt}{\mathtt{t}}
\newcommand{\stt}{\mathtt{s}}
\newcommand{\xtt}{\mathtt{x}}
\newcommand{\ytt}{\mathtt{y}}

\newcommand{\Rcl}{\mathcal{R}}
\newcommand{\Pcl}{\mathcal{P}}
\newcommand{\Bcl}{\mathcal{B}}

\newcommand{\smp}{\mathrm{sim}}

\newcommand{\col}{\colon}

\newcommand{\scl}{\mathrm{scl}}
\newcommand{\cl}{\mathrm{cl}}
\newcommand{\SCL}{\mathrm{SCL}}

\newcommand{\rp}{\mathrm{rp}}
\newcommand{\Gcl}{\mathcal{G}}

\newcommand{\pt}{\mathrm{pt}}
\newcommand{\RC}{\mathrm{RC}}

\newcommand{\Acl}{\mathcal{A}}

\author{Nicolaus Heuer}
\title{The Full Spectrum of SCL on Recursively Presented Groups}
\address{Department of Mathematics\\
  University of Oxford}
\email[N.~Heuer]{heuer@maths.ox.ac.uk}

\begin{document}

\maketitle

\begin{abstract}
We show that the set $\SCL^{\rp} \subset \R^{\geq 0}$ of stable commutator lengths on recursively presented groups equals the set of 
non-negative right-computable numbers. 
Hence all non-negative algebraic and computable numbers are in $\SCL^{\rp}$ and $\SCL^{\rp}$ is not closed under subtraction. 
We also show that every non-negative real number is the stable commutator length of an element in some infinitely presented small cancellation group.
\end{abstract}

\section{Introduction}
Let $G$ be group and let $G' = [G,G]$ be its commutator subgroup. For an element $g \in G'$ the \emph{commutator length} ($\cl_G(g)$) is the least number of commutators needed to express $g$ as their product. Define the \emph{stable commutator length} 
($\scl_G(g)$) via $\scl_G(g) = \lim_{n \to \infty} \cl_G(g^n)/n$.

Stable commutator length $(\scl)$ is well studied in many classes of groups and has geometric meaning. Let $X$ be a pointed connected topological space with fundamental group $G$ and $\gamma \col S^1 \to X$ be a pointed loop, which represents an element $g \in G$. Then both $\scl_G(g)$ and $\cl_G(g)$ measure the least complexity of a surface needed to bound $\gamma$. The theory of these invariants is developed mostly by Calegari and his coauthors; see  \cite{Calegari}.  

In many classes of finitely presented groups $\scl$ has been extensively studied due to its relationship with rotation numbers, surfaces embeddings and simplicial volume; see Section \ref{subsec:examples of scl}. 
It is known to be rational on certain graphs of groups, in particular on free groups and on Baumslag-Solitair groups \cite{Calegari_rational, Chen-scl}. Every rational number is realised as the stable commutator length of an element in some finitely presented group \cite[Remark 5.20]{Calegari}.
However, there are several classes of finitely presented groups known where $\scl$ is not rational \cite{zhuang}  \cite[Chapter 5]{Calegari}. In all of those cases, $\scl$ is even not algebraic.

Calegari asked \cite[Questions 5.47 and 5.48]{Calegari} which real non-negative numbers arise as the stable commutator lengths of finitely presented groups and whether there are elements in finitely presented groups which have algebraic but non-rational stable commutator length.

We answer those questions for the larger class of recursively presented groups. 
Higman's Embedding Theorem (Theorem \ref{thm:higman embedding}) asserts that these are exactly the finitely generated subgroups of finitely presented groups. 
The set of isomorphism classes of recursively presented groups is countable and thus so is the set $\SCL^{\rp} \subset \R^{\geq 0}$ of stable commutator lengths of recursively presented groups. 

We will fully describe $\SCL^{\rp}$ in terms of its computability. We say that real number $\alpha \in \R$ is \emph{right-computable} if the set $\{ q \in \Q \mid \alpha < q \}$ is recursively enumerable; see Definition \ref{defn:computable and right computable numbers}. The set~$\RC^{\geq 0} \subset \R$ of non-negative right-computable numbers is countable, as there are only countably many Turing machines.

\begin{theorem} \label{theorem:classification of scl on rp groups}
We have that
$$
\SCL^{\rp}  = \RC^{\geq 0}
$$
as subsets of $\R^{\geq 0}$ where 
\begin{eqnarray*}
\SCL^{\rp} &:=& \{ \scl_G(g) \mid g \in G \mbox{, G recursively presented} \} \mbox{, and} \\
\RC^{\geq 0} &:=& \{ \alpha \geq 0 \mid \alpha \mbox{ right computable} \}.
\end{eqnarray*}
Every non-negative computable or algebraic number is in $\SCL^{\rp}$. 
The set $\SCL^{\rp}$ is closed under addition but not under subtraction, i.e.\ there are elements $\alpha, \beta \in \SCL^{\rp}$ such that $\alpha-\beta > 0$ but $\alpha-\beta \not \in \SCL^{\rp}$.
\end{theorem}

It is not hard to show that $\SCL^{\rp} \subset \RC^{\geq 0}$. To show the converse we will construct for an element $\alpha \in \RC^{\geq 0}$ an explicit recursive presentation of a group $\Gcl$ with an element $\ttt \in \Gcl$ such that $\scl_\Gcl(\ttt) = \alpha$. For this we will need the following proposition, which is of its own interest.

\begin{proposition} \label{proposition:approx scl groups}
Let $(m_i)_{i \in \N}, (n_i)_{i \in \N}$ be two sequences of natural numbers such that $n_i < n_{i+1}$ and $\frac{m_i}{n_i} \geq \frac{m_{i+1}}{n_{i+1}}$ for every $i \in \N$. 
Then there is an infinitely presented $C'(1/6)$ small cancellation group $\Gcl = \Gcl((m_i)_{i \in \N}, (n_i)_{i \in \N})$ whose presentation is given explicitly in terms of $(m_i)_{i \in \N}$ and  $(n_i)_{i \in \N}$ and an element $\ttt \in \Gcl$ such that
$$
\scl_{\Gcl}(\ttt) = \lim_{i \to \infty} \frac{m_i}{n_i}.
$$
\end{proposition}
The limit in Proposition \ref{proposition:approx scl groups} exists as the sequence $(\frac{m_i}{n_i})_{i \in \N}$ is decreasing and non-negative. As a corollary to Proposition \ref{proposition:approx scl groups} we obtain an explicit proof of the following result:

\begin{corollary} \label{corollary: small cancellation}
Every non-negative real number arises as the stable commutator length of an element in some infinitely presented $C'(1/6)$ small cancellation group. 
\end{corollary}

It would be interesting to promote the results of Theorem \ref{theorem:classification of scl on rp groups} to finitely presented groups.
There are several ways to embed a recursively presented group $G$ into a finitely presented group~$H$; see Theorem \ref{thm:higman embedding}. 
However, for such an embedding $\Phi \col G \to H$ we just get an upper bound $\scl_H(\Phi(g)) \leq \scl_G(g)$ for $g \in G$; see Proposition \ref{prop:scl basic prop}. 
Thus we ask:
\begin{quest} \label{quest: rp fp groups}
Let $G$ be a recursively presented group and let $g \in G$ be an element. Is there a finitely presented group $H$ and a map $\Phi_g \col G \to H$ such that $\scl_H(\Phi_g(g)) = \scl_G(g)$?
\end{quest}
A positive answer to this question would prove Theorem \ref{theorem:classification of scl on rp groups} for finitely presented groups and answer the questions of Calegari \cite[Questions 5.47 and 5.48]{Calegari}.

\subsection*{Organisation}
This paper is organised as follows. In Section \ref{sec:stable commutator length} we recall well known results about stable commutator length and in Section \ref{sec:tm and rp groups}  we recall results about computablility and recursive presentations.
In Section \ref{sec:van kampen on surfaces} we introduce van Kampen diagrams on surfaces and show how they may be used to estimate $\scl$ in Proposition \ref{prop:admissible maps and van Kampen diagrams}.
In Section \ref{sec:proof of Prop} we will use these techniques to prove Proposition \ref{proposition:approx scl groups}. Theorem \ref{theorem:classification of scl on rp groups} and  Corollary \ref{corollary: small cancellation} will be proved in Section \ref{sec:proof of theorems a,b,d}.

\subsection*{Acknowledgements}
I would like to thank Martin Bridson for many very helpful discussions. Moreover, I would like to thank Joe Chen, Jan Steinebrunner for sharing their insights on stable commutator length and computability. I would like to thank Clara L{\"o}h for her insights on computability and comments on a previous version of this article.

\section{Stable Commutator Length} \label{sec:stable commutator length}

In Section \ref{subsec:defn and basic results scl} we will briefly introduce stable commutator length and state well known results.
In Section  \ref{subsec:examples of scl} we recall known results about stable commutator length in finitely presented groups. In Section \ref{subsec:scl and simvol} we state the relationship to simplicial volume.
The reference for stable commutator length is the book of Calegari \cite{Calegari}.

\subsection{Stable Commutator Length: Preliminaries} \label{subsec:defn and basic results scl}
Let $G$ be a group and let $G'=[G,G]$ be the commutator subgroup. Every element $g \in G'$ is the product of commutators $[x,y] := x y x^{-1} y^{-1}$ and we define the \emph{commutator length} of $g$ in $G'$ to be
$$
\cl_G(g) := \min \{n \mid \exists_{x_1, \ldots, x_n, y_1, \ldots, y_n} g = [x_1, y_1] \cdots [x_n,y_n] \}.
$$ 
Similarly, define the \emph{stable commutator length} ($\scl(g)$) of $g$ as
$$
\scl_G(g) := \lim_{n \to \infty} \frac{\cl_G(g^n)}{n}.
$$
It is useful to extend this invariant also for elements of the whole group $G$. If $g \in G$ is an element such that there is an $N$ with $g^N \in G'$ then define $\scl_G(g) := \frac{\scl_G(g^N)}{N}$. Else, set $\scl_G(g) = \infty$. 

There is a geometric interpretation of $\scl$.
Let $X$ be a topological space with fundamental group $\pi_1(X)$ and let $\gamma \col S^1 \to X$ be a loop.
Let $f$ be map $f \col \Sigma \to X$ from a compact oriented surface $\Sigma$ with boundary $\partial \Sigma$ to $X$.
Then the pair $(f, \Sigma)$ is called \emph{admissible to $\gamma$} if there is a commutative diagram
$$
\begin{tikzcd}
\partial \Sigma \arrow[r, "\iota"] \arrow[d, "\partial f"] & \Sigma \arrow[d, "f"] \\
 S^1 \arrow[r, "\gamma"] & X
\end{tikzcd} 
$$
where $\iota \col \partial \Sigma \to \Sigma$ is the natural inclusion of the boundary.
Observe that every component of $\partial \Sigma$ has a orientation induced by the orientation of $\Sigma$. The \emph{degree} $n(f, \Sigma)$ of $(f, \Sigma)$ is defined via $(\partial f)_*[\partial \Sigma] = n(f, \Sigma) [S^1 ]$ in  $H_1(S^1; \Z)$. 
An admissible map is called \emph{positive} if the degree of $\partial f$ on every component of $\partial \Sigma$ is positive. 
If $\Sigma = \sqcup_{i=1}^n \Sigma_i$ where $\Sigma_i$ is connected, we define
$$
\chi^-(\Sigma) = \sum_{i=1}^n \min \{ \chi(\Sigma_i), 0 \},
$$ 
where $\chi$ is the ordinary Euler characteristic.
We may use positive admissible pairs $(f,\Sigma)$ to compute stable commutator length:
\begin{prop}[$\scl$ via admissible maps, \protect{\cite[Proposition 2.10]{Calegari}}] \label{prop:scl via admissible maps}
Let $X$ be a topological space with fundamental group $\pi_1(X)$ and let $\gamma \col S^1 \to X$ be a loop representing an element~$g \in \pi_1(X)$. Then
$$
\scl_{\pi_1(X)}(g) = \inf_{(f,\Sigma) } \frac{-\chi^-(\Sigma)}{2 n(f, \Sigma)}
$$
where the infimum is taken over all positive admissible maps $(f, \Sigma)$ for $\gamma$.
\end{prop}

We collect some well-known results about stable commutator length:

\begin{prop} \label{prop:scl basic prop}
Let $G$ and $H$ be groups, let $g,h \in G$ be elements and let $\Phi \col G \to H$ be a homomorphism. Then
\begin{itemize}
\item[(i)] $\scl_G(g) \geq \scl_H(\Phi(g))$,
\item[(ii)] $\scl_G(gh) \leq \scl_G(g) + \scl_G(h) + \frac{1}{2}$,
\item[(iii)]  $\scl_G([g,h]) \leq \frac{1}{2}$,
\item[(iv)]  $\scl_G(g^n) = n \cdot \scl_G(g)$, for every $n \in \N$, and 
\item[(v)] $\scl_G(g^{-1}) = \scl_G(g)$. 
\end{itemize} 
\end{prop}
See \cite[Lemmata 2.4, 2.24, 2.75, 2.76]{Calegari} for a proof of these statements.

\subsection{Stable Commutator Length in Finitely Presented Groups} \label{subsec:examples of scl}

Many classes of finitely presented groups are known to have \emph{rational} stable commutator length.
Calegari \cite{Calegari_rational} showed that free groups have this property and that $\scl$ may be computed efficiently in free groups in polynomial time. Lvzhou Chen \cite{Chen-scl} gernalised rationality to certain graphs of groups, including Baumslag Solitair groups.

Another class of groups where the stable commutator length is well understood comes from actions of groups on the circle. Let $G$ be a group with vanishing stable commutator length acting on a circle $\rho \col G \to \mathrm{Homeo}^+(S^1)$. Then the stable commutator length of the central extension~$\tilde{G}$ of $G$ associated to the Euler class of $\rho$ 
can be fully understood via rotation numbers \cite[Section 5.2]{Calegari}.
Zhuang \cite{zhuang} used this construction to provide the first example of finitely presented groups with non-algebraic stable commutator length,
see also \cite[Chapter 5.2]{Calegari}.
In all such examples the stable commutator length is non-algebraic unless it is rational \cite[Question 5.48]{Calegari}.

\subsection{Stable Commutator Length and Simplicial Volume} \label{subsec:scl and simvol}
Stable commutator length in finitely presented groups may be used to construct $4$-manifolds with controlled simplicial volume. For every finitely presented group $G$ with $H_2(G;\R) = 0$ and $g \in G'$ there is an orientable closed connected $4$-manifold $M$ such that $\| M \| = 48 \cdot \scl_G(g)$ \cite{Heuer-Loeh}. 
By applying these results to certain finitely presented groups related to the examples of Section \ref{subsec:examples of scl}, the authors could show that every rational number arises as the simplicial volume of a $4$-manifold. 
Thus, Theorem~\ref{theorem:classification of scl on rp groups} and Question \ref{quest: rp fp groups} also indicate which non-negative real numbers arise as the simplicial volume of $4$-manifolds.

\section{Turing Machines and Recursively Presented Groups} \label{sec:tm and rp groups}
We recall well-known results on computability and recursive presentation in Section \ref{subsec:recursive sets and presentations}. 
Encoding such sets in the rational numbers gives rise to computable and right-computable numbers which are discussed in Section \ref{subsec:computable and right computable numbers}.

\subsection{Recursive Sets and Recursive Presentations} \label{subsec:recursive sets and presentations}
Let $\Acl$ be a set which may be encoded as the input set of some Turing machine. Then a subset $\Bcl \subset \Acl$ is called \emph{computable} if there is some Turing machine which decides if an element $a \in \Acl$ lies in $\Bcl$ or not.
A subset $\Bcl \subset \Acl$ is called \emph{recusively enumerable} if there is a Turing machine which enumerates all elements in $\Bcl$. Equivalently there is a Turing machine with input $\Acl$ and whose halting set is $\Bcl$. 
See \cite{Turing} for Turings original paper.

Let $S$ be an alphabet. Then we may interpret $F(S)$ as the input set of some Turing machine and thus we may define recursively enumberable subsets $\Bcl \subset F(S)$.
We say that a presentation $\langle S \mid \Rcl \rangle$ is \emph{recursive} if and only if $\Rcl \subset F(S)$ is recursively enumerable. A group $G$ is \emph{recursively presented} if it admits a recursive presentation.
As there are only countably many Turing machines there are only countably many recursively presented groups.

Let $G$ be a finitely presented group with finite presentation $\langle S \mid \Rcl \rangle$. 
A word in $F(S)$ represents the trivial element in $G$ if and only if it is the product of conjugates of elements of $\Rcl^\pm$ in $F(S)$. We may enumerate all such elements and thus the set of elements $w \in F(S)$ that represent the identity in $G$ is recursively enumerable. This shows that every finitely generated subgroup of a finitely presented group is recursively presented.

Higman showed that indeed the opposite is true:
\begin{thm}[\protect{Higman's Embedding Theorem, \cite{Higman}}] \label{thm:higman embedding}
A finitely generated group $G$ is recursively presented if and only if it is a finitely generated subgroup $G < H$ of a finitely presented group $H$.
\end{thm} 
Indeed there is a \emph{universal} finitely presented group $H$ containing all finitely generated, recursively presented groups as its subgroups.

\subsection{Computable and Right-Computable Numbers}
\label{subsec:computable and right computable numbers}
We can now define computable and right-computable numbers.

\begin{defn}[Computable and Right-Computable Numbers] \label{defn:computable and right computable numbers}
A real number $\alpha \in \R$ is called \emph{computable} if $\{ q \in \Q \mid \alpha < q \}$ is computable and \emph{right-computable} if this set is recursively enumberable. We denote by $\RC \subset \R$ the set of all right-computable real numbers and by $\RC^{\geq 0}$ the set of all such non-negative numbers.
\end{defn}

\begin{exmp} \label{exmp:right computable numbers}
Let $\Acl \subset \N$ be a subset. Then it was observed by Specker \cite{Specker} that the number
$$
\alpha_{\Acl} = \sum_{i \not \in \Acl} 2^{-i}
$$
is computable if and only $\Acl$ is computable as a subset of $\N$ and right-computable if and only if $\Acl$ is recursively enumerable.
It is well known that any encoding of the halting set is recursively enumberable but not computable. Hence there are right-computable numbers that are not computable.
\end{exmp}

\begin{prop}[Properties of $\RC^{\geq 0}$] \label{prop:properties of rc}
Every non-negative algebraic number and every computable number is in $\RC^{\geq 0}$. However, there are numbers in $\RC^{\geq 0}$ that are not computable.
The set $\RC^{\geq 0}$ is closed under addition but not under subtraction, i.e.\ there are elements $\alpha, \beta \in \RC^{\geq 0}$ such that $\alpha - \beta > 0$ but $\alpha - \beta \not \in \RC^{\geq 0}$.
\end{prop}

\begin{proof}
From the definition is it clear that every computable number is right-computable. It is well-known that every algebraic number is computable by employing any type of approximation algorithm to the roots of polynomials.
Example \ref{exmp:right computable numbers} provides a right-computable number $0 < \alpha < 1$ that is not computable.
As $\alpha$ is recursively enumerable, the set
$\{ q \in \Q \mid \alpha < q \}$ is recursively enumerable. 
Suppose that $1-\alpha$ is right-computable. Then also the set $\{ q \in \Q \mid \alpha < q \}$ is recursively enumerable and thus we may decide if a rational number $q \in \Q$ satisfies $\alpha < q$ or not.
Then $\alpha$ would be computable, which is a contradiction. Thus $\RC^{\geq 0}$ is not closed under subtraction. However, it is easy to see that $\RC^{\geq 0}$ is indeed closed under addition.
\end{proof}

For the proof of Proposition \ref{proposition:approx scl groups} we will use the following equivalent characterisation of right-computable numbers:

\begin{lemma} \label{lemma:characterisation of right-computable}
A number $\alpha \in \R^{\geq 0}$ is right-computable if and only if there is a Turing machine~$T$ with input set $\N$ and output set $\N \times \N$ (appropriately encoded), which for every $i \in \N$ returns a pair $T(i) = (m_i,n_i)$ of natural numbers and which satisfies that for all $i \in \N$, $n_i < n_{i+1}$ and
$$
\frac{m_i}{n_i} \geq \frac{m_{i+1}}{n_{i+1}} 
$$
such that
$$
\alpha = \lim_{i \to \N} \frac{m_i}{n_i}.
$$
\end{lemma}
Observe that the limit in the last lemma exists as the sequence $(\frac{m_i}{n_i})_{i \in \N}$ is monotone, bounded, and positive.

\begin{proof}
Let $T$ be the turing machine which halts if $q \in \Q$ satisfies that $\alpha < q$. We may compute a list all elements in $\Q$, and let $T$ run on each of them. Thus, 
we may compute a sequence $(q_i)_{i \in \N}$ of all elements $\{ q \in \Q \mid  \alpha < q \}$ and we may restrict this sequence to a strictly decreasing sequence $\tilde{q}_i = \min \{ q_j \mid j = 1, \ldots, i \}$ whose limit is $\alpha$. We may assume that $\tilde{q}_i = \frac{m_i}{n_i}$ and $n_i < n_{i+1}$ by possibly scaling $m_i$ and $n_i$.
\end{proof}

\section{Van Kampen Diagrams on Surfaces} \label{sec:van kampen on surfaces}

We introduce van Kampen Diagrams on surfaces, which we will use to encode the admissible maps $(f, \Sigma)$ of Proposition~\ref{prop:scl via admissible maps}. We will estimate the Euler characteristic of such surfaces by defining a combinatorial curvature $\kappa(P)$ for the disks $P$ of a van Kampen diagram in Section~\ref{subsec:comb gauss bonnet}. We will estimate $\kappa(P)$ via the \emph{branch vertices} of $P$, introduced in Section \ref{subsec:branch vertices}.

\subsection{Admissible Surfaces via Van Kampen Diagrams}
Van Kampen Diagrams on surfaces have been introduced by Olshanskii to study homomorphisms from surface groups to a group with a fixed presentation \cite{OL}. See also \cite{CSS}.
Let $\Sigma$ be a compact surface with boundary $\partial \Sigma$ and let $\langle \Scl \mid \Rcl \rangle$ be a presentation of a group $G$. A \emph{van Kampen Diagram~$\Pcl$ on~$\Sigma$ over the presentation of $\langle S \mid \Rcl \rangle$} is a decomposition of~$\Sigma$ into finitely many polygons $(P_i)_{i \in I}$ called \emph{disks} where the edges are labelled by words in $\Scl$ such that the boundary of each polygon is labelled counterclockwise by a reduced element of $\Rcl^\pm$ and such that the edges of adjacent polygons are compatible, i.e.\ if an edge is adjacent to two polygons $P$ and $P'$ then the label of one edge is $w \in F(S)$ and the other one is $w^{-1}$.

Every van Kampen diagram on a surface yields a continuous map $f \col \Sigma \to X$ to the presentation complex $X$ of $G = \langle \Scl \mid \Rcl \rangle$ by mapping the labelled edges to the edges in the $1$-skeleton of $X$ and mapping the polygons to the corresponding $2$-cells of $X$.

Van Kampen diagrams may be used to estimate $\scl$ of elements in $G$. 

\begin{prop} \label{prop:admissible maps and van Kampen diagrams}
Let $G$ be a group with presentation $\langle \Scl \mid \Rcl \rangle$ and assume that every element of~$\Rcl$ is cyclically reduced. Let $X$ be the associated presentation complex and let $\ttt \in \Scl$ be a letter represented by a loop $\gamma \col S^1 \to X$.
Let $f \col \Sigma \to X$ be a positive admissible surface to $\gamma$.

Then there is a van Kampen diagram $(\tilde{f}, \tilde{\Sigma})$ on a surface $\tilde{\Sigma}$ such that the boundaries are labelled by positive powers of $\ttt$, such that $(\tilde{f},\tilde{\Sigma})$ has the same degree as $(f,\Sigma)$ and such that
$$
\chi^-(\Sigma) \leq \chi^-(\tilde{\Sigma}).
$$
\end{prop}

\begin{proof}
Let $f \col \Sigma \to X$ be an admissible map. We may assume that $\gamma$ maps the boundary to the edge in the $1$-skeleton of $X$ labelled by $\ttt$.

Both $X$ and $\Sigma$ admit the structure $X^{\smp}$ and $\Sigma^{\smp}$ of a simplicial complex. Let $\pt^{\smp} \in X^{\smp}$ be the corresponding base-point of $X$. 
Every element $\stt \in \Scl$ corresponds to a simplicial loop in $X$. Let $\stt^{\smp}$ be the corresponding loop in $X^{\smp}$  and denote by $\Scl^{\smp}$ the set of all such simplicial loops.
Moreover, let $\Rcl^{\smp}$ be the set  of $2$-subcomplexes $r^{\smp} \subset X^{\smp}$ of $X^{\smp}$ corresponding to a relation $r  \in \Rcl$.  
By simplicial approximation we may assume that $f$ is homotopic to a simplicial map $f^{\smp}$.
By possibly further subdividing $\Sigma^{\smp}$ we may assume that the preimage of every element in $\Scl^{\smp}$ consists of $1$-simplices without back-tracking.

By possibly subdividing $\Sigma^{\smp}$ and changing $f^{\smp}$ further we may assume that $(f^\smp )^{-1}(\pt)$ is a finite collection of points.
Let $\stt \in \Scl$ and consider the inverse
 $(f^{\smp})^{-1}(\stt^\smp_o)$ of the interior of $\stt^\smp_o$ of $\stt^{\smp}$. By possibly subdividing $\Sigma^\smp$ futher we may assume that  $(f^{\smp})^{-1}(\stt^\smp_o)$ is a set of $1$-simplices, connecting the points of $(f^\smp)^{-1}(\pt^\smp)$.
 As the boundary of $\Sigma^\smp$ maps to $\ttt^\smp$, we know that every edge in $(f^\smp)^{-1}(\stt^\smp_o)$ bounds two polygons on either side if $\stt \not = \ttt$. 
 
 Any edge of $(f^\smp)^{-1}(\ttt_o)$ that does not bound any polygon may be removed without affecting the total degree and by possibly increasing the Euler characteristic. Thus we may assume that every edge of $(f^\smp)^{-1}(\ttt^{\smp})$ is adjacent to either one or two polygons.

Let $r \in \Rcl$ be a relation and let $r^{\smp} \in \Rcl^{\smp}$ be the corresponding subcomplex of $X^{\smp}$.
The preimage $\Pcl_r$ of the interior of $r^{\smp}$ under $f^{\smp}$ is a disjoint union of open subcomplexes of $\Sigma^{\smp}$ whose boundary is the preimage of $\Scl^{\smp}$. 
By possibly reducing the Euler characteristic of $\Sigma^{\smp}$ we may assume that every component of $\Pcl_r$ is homeomorphic to a disk.
Let $P \in \Pcl_r$ be such a disk component. Its boundary under $f^{\smp}$ maps to the loops in $\Scl^{\smp}$ and thus defines some conjugacy class in $F(\Scl)$, which is represented by a power $n_P \in \Z$ of $r$.
If $n_P = 0$, we may remove~$P$, possibly by subdividing $\Sigma^{\smp}$ further.

If $|n_P| > 1$ we may replace $P$ by $|n_P|$ many disks $P_1, \ldots, P_{|n_P|}$, which are labelled by a power of $r$ of the same sign of $n_P$; see Figure \ref{fig:splitting}.

\begin{figure} 
\begin{center}
\includegraphics[scale=0.7]{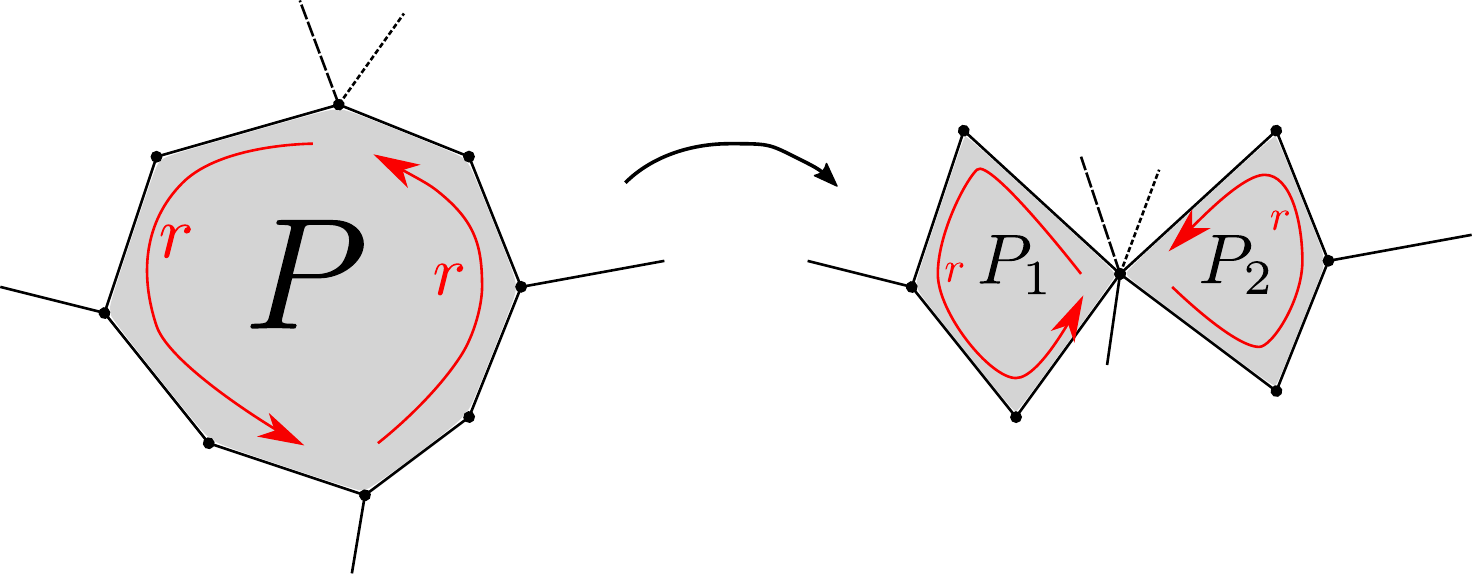}
\caption{$P$ has the boundary word $r^2$. Replace $P$ by two disk components $P_1$ and $P_2$ which both have the boundary word $r$.} \label{fig:splitting}
\end{center}
\end{figure}

Hence, we may assume that for every $r \in \Rcl$, the preimage $\Pcl_r$ of the interior of $r^{\smp}$ is labelled by a word which is either conjugate to $r$ or $r^{-1}$. Moreover, since $\Sigma$ is compact, there are just finitely many $r \in \Rcl$ such that $\Pcl_r$ is non-empty.
We may assume that there is no backtracking on the paths of the boundary of~$P$ by glueing together two paths with backtracking; see Figure~\ref{fig:paste}.

\begin{figure} 
\begin{center}
\includegraphics[scale=0.7]{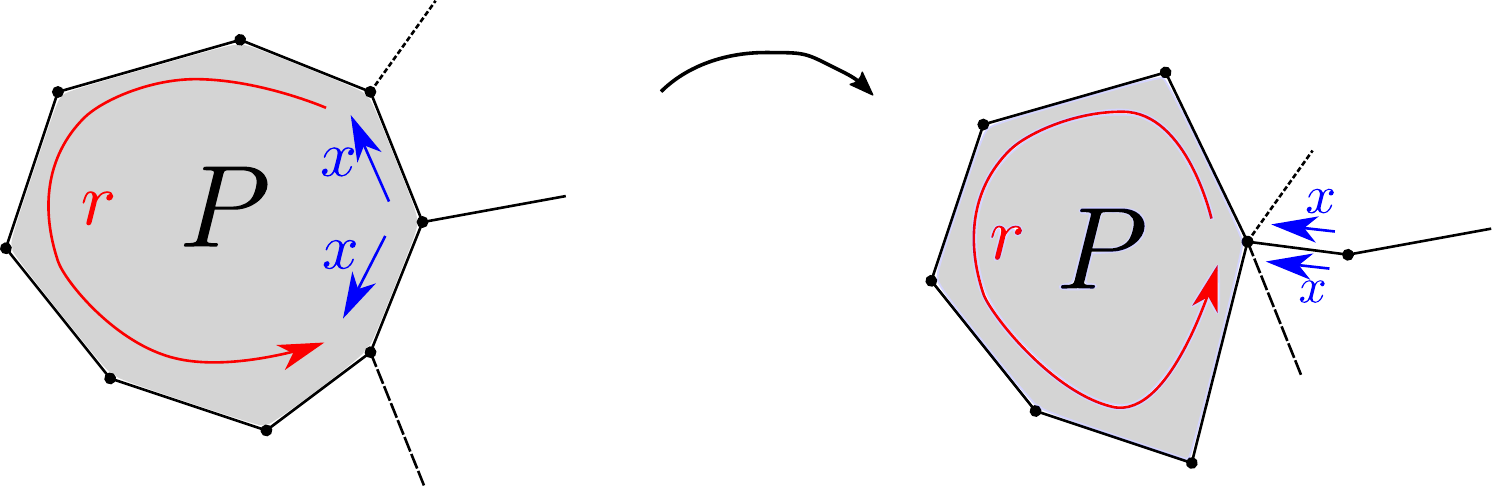}
\caption{If the boundary word of~$P$ has backtracking we may glue up the backtracking and replace it by a disk with shorter boundary word.} \label{fig:paste}
\end{center}
\end{figure}

We conclude that every $P$ is a polygon labelled by a positive or a negative power of $r \in \Rcl$ without backtracking.
\end{proof}

\subsection{Combinatorial Gauss-Bonnet} \label{subsec:comb gauss bonnet}
For a van Kampen diagram, let $\Pcl = \{ P_i \}_{i \in I}$ be the set of labelled polygons and let $P \in \Pcl$. 
For any vertex $v$ of $P$ denote by $\deg(v)$ be the degree of~$v$ i.e.\ the number of edges adjacent to $v$ and let $\deg'(v)$ be the number of disks adjacent to $v$.  
Observe that $2 \leq \deg'(v) \leq \deg(v)$ and that by concatenating edges we may always assume that~$\deg(v) \geq 3$.

Let $V_P$ be the set of all vertices and $E_P$ be  the set of all edges of $P$.
For every edge $e \in E_P$ let $\epsilon(e) \in \{ 1, 2 \}$ be the number of disks adjacent to it.

Define the \emph{curvature} $\kappa \col \Pcl \to \Q$ of disks via
$$
\kappa(P) := \left( \sum_{v \in V_P} \frac{1}{\deg'(v)} \right) - \left( \sum_{e \in E_P} \frac{1}{\epsilon(e)} \right) + 1.
$$

\begin{prop}[Combinatorial Gauss-Bonnet] \label{prop:combinatorial gauss bonnet}
Let $\Pcl$ be a van Kampen diagram on a surface~$\Sigma_\Pcl$ and let $\kappa \col \Pcl \to \Q$ be as above. Then
$$
\chi(\Sigma_\Pcl) = \sum_{P \in \Pcl} \kappa(P).
$$
\end{prop}

\begin{proof}
Every vertex in $\Sigma_\Pcl$ is adjacent to $\deg'(v)$ many disks. Thus $\sum_{P \in \Pcl} \sum_{v \in V_P} \frac{1}{\deg'(v)}$ is the total number of vertices in $\Sigma_\Pcl$. Similarly, $\sum_{P \in \Pcl} \sum_{e \in E_P} \frac{1}{\epsilon(v)}$ is the total number of edges (as every edge is the boundary of some disk) and $\sum_{P \in \Pcl} 1$ is the total number of faces.
Hence
$$
\sum_{P \in \Pcl} \kappa(P) = \# \{ \mbox{vertices} \} - \# \{ \mbox{edges} \} + \# \{ \mbox{faces} \} = \chi(\Sigma). \qedhere
$$
\end{proof}

\subsection{Curvature via Branch Vertices} \label{subsec:branch vertices}
We will estimate the curvature of van Kampen diagrams via the number of \emph{branch vertices} of a disk.
Let $P \in \Pcl$ be a disk with vertices $V_P$ and edges $E_P$.
For $v \in V_P$ set
$$
\beta(v) =
\begin{cases}
1 & \mbox{if } \deg(v) \geq 3 \\
0 & \mbox{else.}
\end{cases}
$$
We call a vertex $v$ such that $\beta(v) = 1$ a \emph{branch vertex.}

\begin{defn}[Branch Vertices $\beta$] \label{defn:beta}
Let $P \in \Pcl$ be a disk in a van Kampen diagram and let $p \subset \partial P$ be a connected subpath in the boundary of $P$ with vertices $v_0, \ldots, v_n$.
Then we set
$$
\beta(P) := \sum_{v \in V_P} \beta(v)
$$
and
$$
\beta(p) = \frac{1}{2} \left( \beta(v_0) + \beta(v_n) \right) + \sum_{i=1}^{n-1} \beta(v_i).
$$
\end{defn}

The following observation will be useful:
\begin{prop} \label{prop:useful property beta}
Let $\Pcl$ be a van Kampen diagram on a surface $\Sigma$, let $P \in \Pcl$ be a disk and let $p,q$ bet two subpath of $\partial P$ such that $q$ starts at the end of $q$. Then
$$
\beta(p \cdot q) = \beta(p) + \beta(q).
$$
where we denote by $p \cdot q$ the concatenation of boths paths.
If $\partial P = p_1 \cdots p_n$ is a decomposition of the boundary of $P$ into subpaths, then
$$
\beta(P) = \sum_{i=1}^n \beta(p_i).
$$
\end{prop}
\begin{proof}
Immediate from the definition of $\beta$.
\end{proof}

We may use $\beta(P)$ to estimate the curvature of $P$.

\begin{prop} \label{prop:curvature via branch vertices}
Let $\Pcl$ be a van Kampen diagram on a surface $\Sigma$ and let $P \in \Pcl$ be a disk with vertices $V_P$. Then
$$
-\kappa(P) \geq \frac{ \beta(P) - 6 }{6} 
$$
\end{prop}

\begin{proof}
Immediate from Proposition \ref{prop:combinatorial gauss bonnet}.
\end{proof}

\section{Proof of Proposition \ref{proposition:approx scl groups}}
\label{sec:proof of Prop}

We now define the groups $\Gcl((m_i)_{i \in \N}, (n_i)_{i \in \N})$ of Proposition \ref{proposition:approx scl groups}. 

\begin{defn}[Groups $\Gcl((m_i)_{i \in \N}, (n_i)_{i \in \N})$] \label{defn:groups constructed}
Set $\Scl = \{ \ttt, \att, \btt, \ctt, \stt_1, \stt_2, \stt_3, \stt_4, \stt_5, \stt_6, \stt_7, \stt_8, \stt_9 \}$ and for integers $n,m,N \in \N$ define $r_{n,m,N} \in F(\Scl)$ as
$$
r_{m,n,N} = \ttt^n w_{N}^{2m} s_{N,m,n} 
$$
where
$$
w_N = \att^N \btt^N \ctt^N \att^{-N} \btt^{-N} \ctt^{-N}
$$ and 
$$
s_{N,m,n} = \stt_1^{l} \stt_2^l \stt_3^l \stt_1^{-l} \stt_2^{-l} \stt_3^{-l} \stt_4^{l} \stt_5^l \stt_6^l \stt_4^{-l} \stt_5^{-l} \stt_6^{-l} \stt_7^{l} \stt_8^l \stt_9^l \stt_7^{-l} \stt_8^{-l} \stt_9^{-l}
$$
with $l = 6 \cdot 5^N 7^m 11^n$.
For integer sequences $(m_i)_{i \in \N}$, $(n_i)_{i \in \N}$ define
$$
\Gcl((m_i)_{i \in \N}, (n_i)_{i \in \N}) = \langle \Scl \mid \{ r_{m_i,n_i,i}, i \in \N \} \rangle.
$$
\end{defn}

Note that $l$ is chosen in this way so that $r_{n,m,N}$ satisfy the small cancellation condition $C'(1/6)$ and so that every choice of $n,m,N$ yields a different $l$.

We may now prove Proposition \ref{proposition:approx scl groups}:

\begin{repproposition}{proposition:approx scl groups}
Let $(m_i)_{i \in \N}, (n_i)_{i \in \N}$ be two sequences of natural numbers such that $n_i < n_{i+1}$ and such that $\frac{m_i}{n_i} \geq \frac{m_{i+1}}{n_{i+1}}$ for every $i \in \N$. 
Then the group $\Gcl = \Gcl((m_i)_{i \in \N}, (n_i)_{i \in \N})$
satisfies the $C'(1/6)$ small cancellation condition. 
For $\ttt \in \Gcl$ we have that
$$
\scl_{\Gcl}(\ttt) = \lim_{i \to \infty} \frac{m_i}{n_i}.
$$
If $(m_i)_{i \in \N}$ and $(n_i)_{i \in \N}$ are recursively enumerable then $\Gcl$ is recursively presented.
\end{repproposition}

\begin{proof}
We will estimate $\scl_\Gcl(\ttt)$ from above and below. 
First, note that for every $i \in \N$, 
$$
t^{n_i} = s_N w_N^{-2m_i}.
$$ 
in $\Gcl$.
Moreover, we observe that $w_N$ is a commutator, as $w_N = [\att^N \btt^N, \ctt^N \att^{-N}]$, and that $s_N$ is the product of three commutators, as $s_N = [\stt^l_1 \stt^l_2, \stt^l_3 \stt_1^{-l}] [\stt^l_4 \stt^l_5, \stt^l_6 \stt_4^{-l}] [\stt^l_7 \stt^l_8, \stt^l_9 \stt_7^{-l}] $.
Using Proposition~\ref{prop:scl basic prop} $(ii), (iii)$ and $(iv)$ we see that
$$
\scl_\Gcl(\ttt^{n_i}) \leq 3 + m_i
$$
hence $\scl_\Gcl(\ttt) \leq \frac{m_i+3}{n_i}$ and hence 
$$
\scl_\Gcl(\ttt) \leq \lim_{i \to \infty} \frac{m_i}{n_i},
$$
as $n_i \to \infty$ for $i \to \infty$. 

For the other direction fix an $\epsilon > 0$. We know by Proposition \ref{prop:scl via admissible maps} that there is a positive admissible surface $f \col \Sigma \to X$ where $X$ is the presentation complex of $\Gcl$ with the presentation of Definition \ref{defn:groups constructed} such that
$$
\scl_\Gcl(\ttt) \geq \frac{- \chi(\Sigma)}{2 n(f,\Sigma)} - \epsilon.
$$
By Proposition \ref{prop:admissible maps and van Kampen diagrams} we may assume that there is a van Kampen diagram $\Pcl = (P_i)_{i \in I}$ on $\Sigma$ over the presentation of $\Gcl$ whose boundaries are positive powers of $\ttt$.

We say that a disk $P \in \Pcl$ is \emph{positive} if it is labelled by a word in $\{ r_{m_i,n_i,i}, i \in \N \}$ and \emph{negative} if it is labelled by a word in $\{ r_{m_i,n_i,i}^{-1}, i \in \N \}$.
The set of positive disks is denoted by $\Pcl^+$ and the set of negative disks by $\Pcl^-$.

If $P$ is positive and labelled by the word $r_{m_i,n_i,i}$, then we set $n(P) = n_i$ and $m(P) = m_i$. 
Similarly, if $P$ is negative and labelled by the word $r_{m_i,n_i,i}^{-1}$ then we set $n(P) = -n_i$ and $m(P) = m_i$.
We observe that $\sum_{P \in \Pcl} n(P) = n(f,\Sigma)$, the total degree of the map induced by the van Kampen diagram.
We also define $N(P) = i$.

Suppose that $P \in \Pcl^+$ is positive and labelled by $r_{m,n,N}$.
We decompose the boundary of $P$ as $\partial P = p^{(t)} \cdot p^{(w)} \cdot p^{(s)}$ where $p^{(t)}$ is labelled by $\ttt^{n}$, $p^{(w)}$ is labelled by $w_N^{2m}$, and $p^{(s)}$ is labelled by  $s_{N,m,n}$.
We can decompose the path $p^{(w)}$ further as $p^{(w)} = \prod_{i=1}^{2m} q^i_1 \cdot  q^i_2 \cdot q^i_3 \cdot q^i_4 \cdot q^i_5 \cdot q^i_6$ where
\begin{itemize}
\item  $q^i_1$ is labelled by $\att^N$
\item $q^i_2$ is labelled by $\btt^N$
\item $q^i_3$ is labelled by $\ctt^N$
\item $q^i_4$ is labelled by $\att^{-N}$
\item $q^i_5$ is labelled by $\btt^{-N}$
\item $q^i_6$ is labelled by $\ctt^{-N}$
\end{itemize}
Moreover, $u_0$ will denote the inital vertex of $p^{(w)}$ and $u_1$ will denote the terminal vertex of $p^{(w)}$.

\begin{claim} \label{claim:negative tile binds to $q$}
Suppose that $q^i_j$ has vertices $v_0, \ldots, v_N$, let $v = v_0$ or $v = v_N$ and suppose that $\deg(v) = 2$. Let $\tilde{P} \in \Pcl$ be the disk adjacent to $v$. Then $\tilde{P}$ is negative.
\end{claim}

\begin{proof}
Let $v$ be such a vertex and assume that $v \not \in \{ u_0, u_1 \}$.

Let $e_1$ be the edge with endpoint $v$ and let $e_2$ be the edge with inital vertex $v$. Let $\xtt$ be the last letter in the label of $e_1$ and let $\ytt$ be the first letter in the label of $e_2$. Then $\xtt \ytt$ is one of $\{ \att \btt, \btt \ctt, \ctt \att^{-1}, \att^{-1} \btt^{-1}, \btt^{-1} \ctt^{-1}, \ctt^{-1} \att \}$.

The disk adjacent to $v$ has to be labelled by the inverse of this label, hence has to contain one of $\{ \btt^{-1} \att^{-1}, \ctt^{-1} \btt^{-1}, \att \ctt^{-1}, \btt \att, \ctt \btt, \att^{-1} \ctt \}$. 
The only disks which contain such subwords are negative disks.

Similarly, if $v = u_0$ and $\deg(v) = 2$, then $P$ contains the subword $\ttt \att$ and hence $\tilde{P}$ has to contain the subword $\att^{-1} \ttt^{-1}$ and thus is negative. The same argument holds for $v = u_1$.
\end{proof}

\begin{defn}[The function $\mu$] \label{defn:mu}
Define $\mu$ on the segments $q^i_j$ as follows:
For a vertex $v \in p^{(w)}$ we set $\mu(v) = 1$ if
\begin{itemize}
\item $v$ is the endpoint (and hence also the start) of some $p^i_j$, and
\item $\deg(v)=2$, and
\item the disk $\tilde{P}$ adjacent to $v$ is labelled by $r_{n',m',N'}^{-1}$ (see Claim \ref{claim:negative tile binds to $q$}) and $N'\geq N$.
\end{itemize}
Else, we set $\mu(v)=0$.
For a segment $q^i_j$ with vertices $v_0, \ldots, v_N$ we set
$$
\mu(q^i_j) = \frac{1}{2} \left( \mu(v_0) + \mu(v_N) \right).
$$
Finally, for $P$ as above set
$$
\mu(P) =
\sum_{i = 1, \ldots, 2m} \sum_{j = 1, \ldots, 6} \mu(q^i_j).
$$
\end{defn}

\begin{claim} \label{claim:estimate kappa and mu}
For every segment $q^i_j$ as above we have that 
\begin{eqnarray} \label{eqn:inequ kappa mu}
\beta(q^i_j) + \mu(q^i_j) \geq 1
\end{eqnarray}
where $\beta$ is as in Definition \ref{defn:beta} and $\mu$ is as in Definition \ref{defn:mu}.
\end{claim}

\begin{proof}
Let $q^i_j$ be the segment with vertices $v_0, \ldots, v_N$. We will distinguish between the following cases: 

\begin{itemize}
\item We have $\deg(v_i) \geq 3$ for some $i \in \{2, \ldots, N-1 \}$: In this case we see that $\beta(q^i_j) \geq 1$ and thus $(\ref{eqn:inequ kappa mu})$ follows.
\item 
We have $\deg(v_i)=2$ for all $i \in \{2, \ldots, N-1 \}$ and $\deg(v_0), \deg(v_N) \geq 3$.
In this case again $\beta(q^i_j) \geq 1$ and thus (\ref{eqn:inequ kappa mu}) holds.

\item 
We have $\deg(v_i)=2$ for all $i \in \{1, \ldots, N-1 \}$ and $\deg(v_N) \geq 3$.
Suppose that $q^i_j$ is labelled by $\xtt^N$ and that the edge with terminal vertex $v_1$ ends in $\ytt$. Then $\ytt \xtt^N$ is one of
$$
\{ \ttt \att^N, \ctt^{-1} \att^N, \att \btt^N, \btt \ctt^N, \ctt \att^{-N}, \att^{-1} \btt^{-N}, \btt^{-1} \ctt^{-N} \}
$$
and thus the disk adjacent to $v$ has to contain the words
$$
\{ \att^{-N} \ttt^{-1}, \att^{-N} \ctt, \btt^{-N} \att^{-1}, \ctt^{-N} \btt^{-1}, \btt^N \att, \ctt^N \btt \}
$$ 
as subwords. The only such disks are labelled by negative relations $r_{m_i,n_i,i}^{-1}$ where $i \geq N$.
Thus, in this case, $\beta(q^i_j) = \frac{1}{2}$ and $\mu(q^i_j) = \frac{1}{2}$ and hence $(\ref{eqn:inequ kappa mu})$ follows.

\item 
We have $\deg(v_i)=2$ for all $i \in \{2, \ldots, N \}$ and $\deg(v_1) \geq 3$: Analogous to the previous case.

\item 
We have $\deg(v_i)=2$ for all $i \in \{1, \ldots, N \}$.
Suppose that $q^i_j$ is labelled by $\xtt^N$ and that the edge with terminal vertex $v_1$ ends in $\ytt_1$ and that the edge with initial vertex $v_N$ starts in $\ytt_2$. Then $\ytt_1 \xtt^N \ytt_2$ is one of
$$
\{ \ttt \att^N \btt, \ctt^{-1} \att^N \btt, \att \btt^N \ctt, \btt \ctt^N \att^{-1}, \ctt \att^{-N} \btt^{-1},  \att^{-1} \btt^{-N} \ctt^{-1}, \btt^{-1} \ctt^{-N} \att \}
$$
and thus the disk adjacent to $v$ has to contain the words
$$
\{ \btt^{-1} \att^{-N} \ttt^{-1}, \btt^{-1} \att^{-N} \ctt, \ctt^{-1} \btt^{-N} \att^{-1}, \att^{-1} \ctt^{-N} \btt^{-1}, \ctt \btt^N \att, \att^{-1} \ctt^N \btt \}
$$ 
as subwords. The only such disks are labelled by negative relations $r_{m,n,N}^{-1}$.
Thus in this case $\mu(q^i_j) = 1$ and $(\ref{eqn:inequ kappa mu})$ follows. \qedhere
\end{itemize}
\end{proof}

\begin{claim} \label{claim:positive polygon, mu and nu}
For a positive disk $P \in \Pcl^+$ we have that
$$
-\kappa(P) \geq - \frac{1}{6} \mu(P)+ 2 m(P)
$$
where $\kappa(P)$ is the curvature (see Section \ref{subsec:comb gauss bonnet}) and $\mu(P)$ is as in Definition \ref{defn:mu}.
\end{claim}

\begin{proof}
Fix a positive disk $P \in \Pcl^+$ labelled by $r_{m,n,N}$ and again assume that its boundary decomposes into $\partial P = p^{(t)} \cdot p^{(w)} \cdot p^{(s)}$ where $p^{(t)}$ is labelled by $\ttt^{n}$, $p^{(w)}$ is labelled by $w_N^{2m}$, and $p^{(s)}$ is labelled by  $s_{N,m,n}$. Moreover, the segment $p^{(w)}$ is the concatenation of the segments $q^i_j$ described above.
We will estimate $-\kappa(P)$  using Proposition \ref{prop:curvature via branch vertices}. 

Using Proposition \ref{prop:useful property beta} we may estimate $\beta(P) \geq \beta(p^{(w)}) + \beta(p^{(s)})$ .
It follows from small cancellation theory that $\beta(p^{(s)}) \geq 7$.
Hence, by Proposition \ref{prop:curvature via branch vertices} we have that
$$
-\kappa(P) \geq \frac{\beta(p^{(w)})}{6}.
$$ 
By Claim \ref{claim:estimate kappa and mu} we have that $\beta(q^i_j) \geq 1 - \mu(q^i_j)$ for every segment $q^i_j$ as above. Note that $p^{(w)}$ consists of $12 m(P)$ such segments. By additivity of $\beta$ (Proposition \ref{prop:useful property beta}) and $\mu$ we conclude that
$$
-\kappa(P) \geq -\frac{1}{6} \mu(P) + 2 m(P). \qedhere
$$
\end{proof}

\begin{claim} \label{claim:estimation of nu}
We have that 
$$
n(\Sigma) \leq \sum_{P \in \Pcl^+} \left( n(P) - \mu(P) \frac{n(P)}{12 m(P)}  \right)
$$
where $n(\Sigma)$ denotes the total degree of $\Sigma$.
\end{claim}

\begin{proof}
Fix a positive polygon $P \in \Pcl^+$. We recall that $\mu(P)$ counts the number of subwords 
\begin{eqnarray} \label{equ:segment}
\btt^{-1} \att^{-1}, \ctt^{-1} \btt^{-1}, \att \ctt^{-1}, \btt \att, \ctt \btt, \att^{-1} \ctt
\end{eqnarray}
of negative disks $\tilde{P}$ adjacent to $P$. 
Fix such a disk $\tilde{P}$. There are $12 m(\tilde{P})$ such above segments and the total degree of $\tilde{P}$ is $n(\tilde{P})$. Thus one segment of (\ref{equ:segment}) contributes $\frac{n(\tilde{P})}{12 m(\tilde{P})}$ to $\sum_{P \in \Pcl^-} n(P)$.
Recall that for a polygon labelled by $r_{m,n,N}$ we set $N(P)=N$.
Since $N(\tilde{P}) \geq N(P)$ we see that 
$$
\frac{n(\tilde{P})}{m(\tilde{P})} \leq -\frac{n(P)}{m(P)}
$$
as we chose the sequence $\left( \frac{n_i}{m_i} \right)_{i \in \N}$ to be decreasing and as $n(\tilde{P})$ is negative.
We conclude that
$$
-\sum_{P \in \Pcl^+} \mu(P) \frac{1}{12} \frac{n(P)}{m(P)} \geq \sum_{P \in \Pcl^-} n(P).
$$
and hence
$$
n(\Sigma) = \sum_{P \in \Pcl^+} n(P)  + \sum_{P \in \Pcl^-} n(P) \leq  \sum_{P \in \Pcl^+} n(P) - \mu(P) \frac{n(P)}{12 m(P)} \qedhere
$$
\end{proof}

We are now able to finish the proof of Proposition \ref{proposition:approx scl groups}.
Let $\Pcl$ be a van Kampen diagram on an admissible surface $\Sigma$ to $\gamma$. We may estimate
$$
\frac{-\chi(\Sigma)}{2 n(\Sigma)}=\frac{\sum_{P \in \Pcl} -\kappa(P)}{2 \sum_{P \in \Pcl} n(P)}.
$$
Observe that for every negative tile $P \in \Pcl^-$ we have that $\kappa(P) \leq 0$ by small cancellation theory, Claim \ref{claim:positive polygon, mu and nu} asserts that
$$
\sum_{P \in \Pcl^+} -\kappa(P) \geq  \sum_{P \in \Pcl^+} - \frac{1}{6} \mu(P)+ 2 m(P).
$$
Moreover, we know by Claim \ref{claim:estimation of nu} that
$$
n(\Sigma) \leq \sum_{P \in \Pcl^+} \left( n(P) - \frac{\mu(P)}{12} \frac{n(P)}{m(P)} \right)
$$
and hence
\begin{eqnarray*}
\frac{-\chi(\Sigma)}{2 n(\Sigma)} &\geq & \frac{\sum_{P \in \Pcl^+} \left( m(P) - \frac{\mu(P)}{12} \right) }{\sum_{P \in \Pcl^+} \left( n(P) - \frac{\mu(P)}{12} \frac{n(P)}{m(P)} \right) } \\
&=& 
\frac{\sum_{P \in \Pcl^+} \left( m(P) - \frac{\mu(P)}{12} \right)}{\sum_{P \in \Pcl^+} \frac{n(P)}{\mu(P)} \left( m(P) - \frac{\mu(P)}{12} \right)}. 
\end{eqnarray*}

It is easy to see that for every $P \in \Pcl^+$ we have that $m(P) - \frac{\mu(P)}{12} \geq 0$.
Let $\Pcl' \subset \Pcl^+$ be the set of positive disks such that $m(P) - \frac{\mu(P)}{12} > 0$.
As
$$
0 <
n(\Sigma) \leq \sum_{P \in \Pcl^+} \frac{n(P)}{\mu(P)}  \left( m(P) - \frac{\mu(P)}{12} \right)
$$
we see that the set $\Pcl'$ is non-empty. Moreover, observe that
$$
\frac{\mu(P)}{n(P)} \geq \lim_{i \to \infty} \frac{m_i}{n_i} =: \alpha
$$
and hence
$$
\frac{- \chi(\Sigma)}{2 n(\Sigma)} \geq 
\frac{\sum_{P \in \Pcl'} \left( m(P) - \frac{\mu(P)}{12} \right)}{\sum_{P \in \Pcl'} \frac{1}{\alpha} \left( m(P) - \frac{\mu(P)}{12} \right)} \geq \alpha = \lim_{i \to \infty} \frac{m_i}{n_i}.
$$

We conclude that every $\epsilon > 0$ we have that
$$
\scl_\Gcl(\ttt) \geq \frac{- \chi(\Sigma)}{2 n(\Sigma)}-\epsilon \geq \lim_{i \to \infty} \frac{m_i}{n_i} - \epsilon,
$$
and thus
$$
\scl_\Gcl(\ttt) \geq \lim_{i \to \infty} \frac{m_i}{n_i}.
$$
Together with the opposite inequality derived above we may conclude an equality and thus prove Proposition \ref{proposition:approx scl groups}. 
\end{proof}

\section{Proof of Theorem \ref{theorem:classification of scl on rp groups} and Corollary \ref{corollary: small cancellation}}
\label{sec:proof of theorems a,b,d}
We can now prove Theorem \ref{theorem:classification of scl on rp groups} and Corollary \ref{corollary: small cancellation}.

\begin{proof}[Proof of Theorem \ref{theorem:classification of scl on rp groups}]
First we show that $\SCL^{\rp} = \RC^{\geq 0}$. We will show both inclusions. 

For $\SCL^{\rp} \subset \RC^{\geq 0}$ fix a recursively presented group $G$ with recursive presentation $G \cong \langle S \mid \Rcl \rangle$ and an element $g \in G$ represented by an element $w \in F(S)$. 
We will show that $\{ q \in \Q \mid \scl_G(g) < q \}$ is the halting set of some Turing machine.

\begin{claim}
For every $N \in \N$ there is a Turing machine $T_N$ with input $\Q^{\geq 0}$ (appropriately encoded) and such that for any $q \in \Q^{\geq 0}$ the machine $T_N$  halts on $q$ if and only if $\cl_G(g^N) \leq q \cdot N$.
\end{claim}
\begin{proof}
Indeed $\cl_G(g^N) \leq q \cdot N$ if and only if 
$\cl_G(g^N) \leq \lfloor q \cdot N \rfloor := q'$.
This happens if and only if there are elements $x_1, \ldots, x_{q'}, y_1, \ldots, y_{q'} \in F(S)$ and an element $r \in R=\langle \langle \Rcl \rangle \rangle$ such that
$$
w^N = [x_1, y_1] \cdots [x_{q'},y_{q'}] \cdot r.
$$
We can enumerate all such solutions and check if the condition is satisfied for one of them.
\end{proof}

Now let $T$ be the Turing machine with input $\Q^{\geq 0}$ which recursively starts running the machines $T_N(q)$ simultaneously for all $N \in \N$ and halts if and only if $T_N(q)$ halts for some $N \in \N$. 
If $\scl_G(g) < q$ then there exists some $N$ such that $\cl_G(g^N) \leq N q$. Thus $T$ halts.

If $\scl_G(g) > q$ then also $\cl_G(g^N) > N q$ for every $N \in \N$ and $T$ does not halt. Thus $T$ has halting set $\{ q \in \Q^{\geq 0} \mid  q > \scl_G(g) \}$ and thus $\scl_G(g)$ right-computable.
Hence $\SCL^{\rp} \subset \RC^{\geq 0}$.

To see $\SCL^{\rp} \supset \RC^{\geq 0}$ let $\alpha \in \RC^{\geq}$ be a right-computable number and let $T$ be the Turing machine of Lemma \ref{lemma:characterisation of right-computable} with input $\N$, output $\N \times \N$, such that $T(i) = (m_i, n_i)$ which satisfies that
$n_i < n_{i+1}$ and $\frac{m_i}{n_i} \leq \frac{m_{i+1}}{n_{i+1}}$ for all $i \in \N$ and such that
$$
\alpha = \lim_{i \to \infty} \frac{m_i}{n_i}.
$$
 Then we may construct the recursive presentation
for the group $\Gcl = \Gcl((m_i)_{i \in \N }, (n_i)_{i \in \N})$ as in Definition \ref{defn:groups constructed}.
Proposition \ref{proposition:approx scl groups} asserts that the stable commutator length of $\ttt \in \Gcl$ is exactly $\alpha$. Thus $\SCL^{\rp} \supset \RC^{\geq 0}$.

The rest of the theorem follows from Proposition \ref{prop:properties of rc}.
\end{proof}

\begin{proof}[Proof of Corollary \ref{corollary: small cancellation}]
Any non-negative real number $\alpha \in \R^{\geq 0 }$ arises as the limit $\lim_{i \to \infty} \frac{m_i}{n_i}$, where $n_i \leq n_{i+1}$ and such that $\frac{m_i}{n_i} \geq \frac{m_{i+1}}{n_{i+1}}$.
We conclude using Proposition \ref{proposition:approx scl groups}. 
\end{proof}

{\small
\bibliographystyle{alpha}
\bibliography{rp_scl}}

\end{document}